\documentclass[12pt,reqno,oneside]{amsproc}
\title[A no-contact result]{A no-contact result for a plate-fluid\\
interaction system in dimension three}

\author[M.~Bukal]{Mario Bukal}
\address{University of Zagreb Faculty of Electrical Engineering and Computing, Unska 3, 10000 Zagreb, Croatia}
\email{mario.bukal@fer.hr}

\author[I.~Kukavica]{Igor Kukavica}
\address{Department of Mathematics, University of Southern California, Los Angeles, CA 90089}
\email{kukavica@usc.edu}

\author[L.~Li]{Linfeng Li}
\address{Department of Mathematics, University of California Los Angeles, Los Angeles, CA 90095}
\email{lli265@math.ucla.edu}

\author[B.~Muha]{Boris Muha}
\address{University of Zagreb, Faculty of Science, Department of Mathematics, Bijeni\v cka cesta 30, 10000 Zagreb, Croatia}
\email{borism@math.hr}

\usepackage[normalem]{ulem}
\usepackage{enumitem}
\usepackage{bbm}
\usepackage{fancyhdr}
\usepackage{comment}
\usepackage{tikz}
\chardef\forshowkeys=0
\chardef\refcheck=0
\chardef\showllabel=0
\chardef\sketches=0
\chardef\showcolors=0

\usepackage{bbm,yfonts}
\usepackage{upgreek}
\usepackage{enumitem}
\usepackage{datetime}
\usepackage{esint}
\usepackage{fancyhdr}
\usepackage{comment}
\usepackage[utf8]{inputenc}
\allowdisplaybreaks

\let\pa\partial   
  
\let\eps\epsilon  

\newcommand{\N}{{\mathbb N}}    
\newcommand{\R}{{\mathbb R}}  

\newcommand{\diver}{\operatorname{div}}

\newcommand{\dist}{\operatorname{dist}}

\newcommand{\D}{\operatorname{D}}

\newcommand{\dd}{{\mathrm{d}}}

\newcommand{\red}{\textcolor{red}}

\newcommand{\nablax}{\nabla_x}
\newcommand{\Deltax}{\Delta_x}

\ifnum\forshowkeys=1

\usepackage[notcite,color]{showkeys}
\fi

\usepackage[margin=1in]{geometry}
\usepackage{amsmath, amsthm, amssymb}
\usepackage{times}
\usepackage{graphicx}
\usepackage{marginnote}
\usepackage[unicode,breaklinks=true,colorlinks=true,linkcolor=black,urlcolor=black,citecolor=black]{hyperref}

\ifnum\refcheck=1
\usepackage{refcheck}
\fi

\begin{document}
\def\Omegaone{\text{1-channel}}
\def\ch{{\hat{c}}}
\def\bp{\bar{\partial}}
\def\bbeta{\beta}
\def\aalpha{\mu}
\def\tt{\mathbb{T}^2}
\def\ttt{t_\ast}
\def\xxx{x_\ast}
\def\iii{I_\ast}
\def\bbb{B_\ast}
\def\aa{\nu}
\def\etah{\eta}
\def\uh{u}
\def\vh{v}
\def\qh{q}
\def\ph{p}
\def\inp{{\colr IN PROGRESS}}
\def\rth{{\colr REVISED UP TO HERE}}
\def\fNS{f}
\def\fP{g}
\def\bnew{\colr {\bf }}
\def\enew{\colb {}}
\def\bold{\colu {\bf }}
\def\eold{\colb {}}
\def\inte{\int_{\Omega_\epsilon}}
\def\intb{\int_{\Gamma_\epsilon}}
\def\D{D}
\def\CommT{{R}}
\def\BndT{{B}}

\def\XX{X}
\def\YY{Y}
\def\ZZZ{Z}

\def\intint{\int    \int}
\def\OO{\mathcal O}
\def\SS{\mathbb S}
\def\CC{\mathbb C}
\def\N{\mathbb N}
\def\RR{\mathbb R}
\def\TT{\mathbb T}
\def\ZZ{\mathbb Z}
\def\HH{\mathbb H}
\def\RSZ{\mathcal R}
\def\LL{\mathcal L}
\def\SL{\LL^1}
\def\ZL{\LL^\infty}
\def\GG{\mathcal G}
\def\erf{\mathrm{Erf}}
\def\mgt#1{\textcolor{magenta}{#1}}
\def\ff{\rho}
\def\gg{G}
\def\sqrtnu{\sqrt{\nu}}
\def\ww{w}
\def\ft#1{#1_\xi}
\def\lec{\lesssim}
\def\les{\lesssim}
\def\gec{\gtrsim}
\renewcommand*{\Re}{\ensuremath{\mathrm{{\mathbb R}e  }}}
\renewcommand*{\Im}{\ensuremath{\mathrm{{\mathbb I}m  }}}

\ifnum\showllabel=1
\def\llabel#1{\marginnote{\color{lightgray}\rm\small(#1)}[-0.0cm]\notag}
\else
\def\llabel#1{\notag}
\fi

\newcommand{\norm}[1]{\left\|#1\right\|}
\newcommand{\nnorm}[1]{\lVert #1\rVert}
\newcommand{\abs}[1]{\left|#1\right|}
\newcommand{\NORM}[1]{| | | #1| | |}

\newtheorem{theorem}{Theorem}[section]
\newtheorem{Theorem}{Theorem}[section]
\newtheorem{corollary}[theorem]{Corollary}
\newtheorem{Corollary}[theorem]{Corollary}
\newtheorem{proposition}[theorem]{Proposition}
\newtheorem{Proposition}[theorem]{Proposition}
\newtheorem{Lemma}[theorem]{Lemma}
\newtheorem{lemma}[theorem]{Lemma}

\theoremstyle{definition}
\newtheorem{definition}{Definition}[section]
\newtheorem{Remark}[theorem]{Remark}

\def\theequation{\thesection.\arabic{equation}}
\numberwithin{equation}{section}

\definecolor{mygray}{rgb}{.6,.6,.6}
\definecolor{myblue}{rgb}{9, 0, 1}
\definecolor{colorforkeys}{rgb}{1.0,0.0,0.0}

\newlength\mytemplen
\newsavebox\mytempbox
\def\weaks{\text{            weakly-* in }}
\def\weak{\text{            weakly in }}
\def\inn{\text{            in }}
\def\cof{\mathop{\rm cof  }\nolimits}
\def\Dn{\frac{\partial}{\partial N}}
\def\Dnn#1{\frac{\partial #1}{\partial N}}
\def\tdb{\tilde{b}}
\def\tda{b}
\def\qqq{u}
\def\lat{\Delta_{\bx}}
\def\biglinem{\vskip0.5truecm\par==========================\par\vskip0.5truecm}

\def\inon#1{\hbox{\ \ \ \ \ \ \ }\hbox{#1}}               
\def\onon#1{\inon{on~$#1$}}
\def\inin#1{\inon{in~$#1$}}

\def\FF{F}
\def\andand{\text{\indeq and\indeq}}
\def\ww{w(y)}
\def\ll{{\color{red}\ell}}
\def\ee{\epsilon_0}
\def\startnewsection#1#2{ \section{#1}\label{#2}\setcounter{equation}{0}}   
\def\nnewpage{ }
\def\sgn{\mathop{\rm sgn  }\nolimits}    
\def\Tr{\mathop{\rm Tr}\nolimits}    
\def\div{\mathop{\rm div}\nolimits}
\def\curl{\mathop{\rm curl}\nolimits}
\def\dist{\mathop{\rm dist}\nolimits}  
\def\supp{\mathop{\rm supp}\nolimits}
\def\indeq{\quad{}}           
\def\period{.}                       
\def\semicolon{  ;}
\def\pa{\partial}            
\def\pt{\partial_t}                

\ifnum\showcolors=1
\def\colr{\color{red}}
\def\colc{\color{cyan}}
\def\colrr{\color{black}}
\def\colb{\color{black}}
\definecolor{colorgggg}{rgb}{0.1,0.5,0.3}
\definecolor{colorllll}{rgb}{0.0,0.7,0.0}
\definecolor{colorhhhh}{rgb}{0.3,0.75,0.4}
\definecolor{colorpppp}{rgb}{0.7,0.0,0.2}
\definecolor{coloroooo}{rgb}{0.45,0.0,0.0}
\definecolor{colorqqqq}{rgb}{0.1,0.7,0}
\def\coly{\color{lightgray}}
\def\colg{\color{colorgggg}}
\def\collg{\color{colorllll}}
\def\cole{}
\def\coll{\color{colorqqqq}}
\def\coleo{\color{colorpppp}}
\def\colu{\color{blue}}
\def\colc{\color{colorhhhh}}
\def\colW{\colb}   
\definecolor{coloraaaa}{rgb}{0.6,0.6,0.6}
\def\colw{\color{coloraaaa}}
\else
\def\colr{\color{black}}
\def\colrr{\color{black}}
\def\colb{\color{black}}
\def\coly{\color{black}}
\def\colg{\color{black}}
\def\collg{\color{black}}
\def\cole{\color{black}}
\def\coleo{\color{black}}
\def\colu{\color{black}}
\def\colc{\color{black}}
\def\colW{\color{black}}
\def\colw{\color{black}}
\fi

\def\comma{ {\rm ,\qquad{}} }            
\def\commaone{ {\rm ,\quad{}} }          
\def\nts#1{{\color{red}\hbox{\bf ~#1~}}} 
\def\ntsf#1{\footnote{\color{colorgggg}\hbox{#1}}} 
\def\blackdot{{\color{red}{\hskip-.0truecm\rule[-1mm]{4mm}{4mm}\hskip.2truecm}}\hskip-.3truecm}
\def\bluedot{{\color{blue}{\hskip-.0truecm\rule[-1mm]{4mm}{4mm}\hskip.2truecm}}\hskip-.3truecm}
\def\purpledot{{\color{colorpppp}{\hskip-.0truecm\rule[-1mm]{4mm}{4mm}\hskip.2truecm}}\hskip-.3truecm}
\def\greendot{{\color{colorgggg}{\hskip-.0truecm\rule[-1mm]{4mm}{4mm}\hskip.2truecm}}\hskip-.3truecm}
\def\cyandot{{\color{cyan}{\hskip-.0truecm\rule[-1mm]{4mm}{4mm}\hskip.2truecm}}\hskip-.3truecm}
\def\reddot{{\color{red}{\hskip-.0truecm\rule[-1mm]{4mm}{4mm}\hskip.2truecm}}\hskip-.3truecm}

\def\tdot{{\color{green}{\hskip-.0truecm\rule[-.5mm]{3mm}{3mm}\hskip.2truecm}}\hskip-.1truecm}
\def\gdot{\greendot}
\def\bdot{\bluedot}
\def\ydot{\cyandot}
\def\rdot{\cyandot}
\def\fractext#1#2{{#1}/{#2}}
\def\ii{\hat\imath}
\def\boris#1{\textcolor{blue}{#1}}
\def\vlad#1{\textcolor{cyan}{#1}}
\def\coli{\color{colorqqqq}}
\def\igor#1{\text{{\textcolor{colorqqqq}{IK: #1}}}}
\def\linfeng#1{\textbf{\textcolor{red}{LL:#1}}}
\def\igorf#1{\footnote{\text{{\textcolor{colorqqqq}{IK: #1}}}}}
\def\mario#1{\textcolor{purple}{MB: #1}}

\newcommand{\myr}[1]{{\color{red} #1 }}%
\newcommand{\uvro}{\upvarrho}
\newcommand{\bv}{\boldsymbol v}
\newcommand{\bV}{\boldsymbol V}
\newcommand{\bu}{\boldsymbol u}
\newcommand{\bs}{\boldsymbol}
\newcommand{\bx}{x}
\newcommand{\obx}{\overline{x}}
\newcommand{\bw}{w}
\newcommand{\bn}{\mathbf n}
\newcommand{\by}{\boldsymbol y}
\newcommand{\bz}{\boldsymbol z}
\newcommand{\bef}{\boldsymbol f}
\newcommand{\material}{\Omega_{\eps,h}}
\newcommand{\vol}{\operatorname{vol}}
\newcommand{\bog}{{\rm Ext}_{\diver}}

\def\AA{Y}
\newcommand{\p}{\partial}
\newcommand{\UE}{U^{\rm E}}
\newcommand{\PE}{P^{\rm E}}
\newcommand{\KP}{K_{\rm P}}
\newcommand{\uNS}{u^{\rm NS}}
\newcommand{\vNS}{v^{\rm NS}}
\newcommand{\pNS}{p^{\rm NS}}
\newcommand{\omegaNS}{\omega^{\rm NS}}
\newcommand{\uE}{u^{\rm E}}
\newcommand{\vE}{v^{\rm E}}
\newcommand{\pE}{p^{\rm E}}
\newcommand{\omegaE}{\omega^{\rm E}}
\newcommand{\ua}{u_{\rm   a}}
\newcommand{\omegaa}{\omega_{\rm   a}}
\newcommand{\ue}{u_{\rm   e}}
\newcommand{\ve}{v_{\rm   e}}
\newcommand{\omegae}{\omega_{\rm e}}
\newcommand{\omegaeic}{\omega_{{\rm e}0}}

\def\red#1{\textcolor{red}{#1}}

\begin{abstract}
We address the fluid-structure interaction between a viscous
incompressible fluid and an elastic plate forming its moving upper
boundary in three dimensions. The fluid is described by the
incompressible Navier-Stokes equations with a free upper boundary that
evolves according to the motion of the structure, coupled via the velocity- and stress-matching conditions. Under the natural energy bounds and
additional regularity assumptions on the weak solutions, we prove a
non-contact property with a uniform separation of the plate from the
rigid boundary. The result does not require damping in the plate
equation.

\end{abstract}

\maketitle

\setcounter{tocdepth}{1} 

\colb
\section{Introduction}
\label{sec01} 
\subsection{Problem formulation}\label{ssec01.1}
In this paper, we study the interaction between a viscous, incompressible fluid and a moving elastic structure forming the fluid’s upper boundary; the structure is modeled either as a purely elastic plate or as a structurally damped plate.
We assume that the fluid domain at a time $t\geq0$, denoted by $\Omega_{\eta(t)}$, is a subgraph of a space-time dependent function $\etah\colon \tt\times[0,+\infty)\to {\mathbb R}$,
which represents the height of the moving upper boundary. 
Here, $\tt$ denotes the two-dimensional torus with side-length~$1$.
Namely, the fluid domain at a time $t\geq 0$ is given by the set 
\begin{equation*}
	\Omega_{\eta (t)}=
	\{  (x, z) : x=(x_1,x_2) \in \tt,
	z\in (0,\eta(\bx,t))\}\subseteq\R^3.
\end{equation*}
The fluid is assumed to be incompressible, Newtonian, and its motion is described by the incompressible Navier-Stokes equations
\begin{align}
\begin{split}
&
\partial_{t} u
- 
\Delta u
+ u\cdot \nabla u
+ 
\nabla p
= f \onon{\Omega_{\eta(t)}\times (0,T)},
\\&
\diver u=0 \onon{\Omega_{\eta(t)}\times (0,T)}
,
\end{split}
\label{EQNS}
\end{align}
where $u\colon \Omega_{\eta(t)} \times (0,T) \to \mathbb{R}^3$, $p\colon \Omega_{\eta(t)} \times (0,T) \to \mathbb{R}$, and $f\colon \Omega_{\eta(t)} \times (0,T) \to \mathbb{R}^3$ denote the velocity, pressure, and external force, respectively. As is standard in the analysis of fluid-structure interaction problems, we use $\Omega_{\eta} \times (0,T)$ as shorthand for $\bigcup_{t \in (0,T)} \Omega_{\eta(t)} \times \{t\}$.
The height of the moving upper boundary $\eta(\bx,t)$ is described by a fourth-order damped plate equation
\begin{equation}
\etah_{tt}
- 
\aa
\Delta_{\bx} \etah_t  
+ 
\Delta_{\bx}^2 \etah
= G \onon{\tt\times (0,T)}
,
\label{EQPL}
\end{equation}
where $\Delta_{\bx}=\partial_{11}+\partial_{22}$ is the horizontal Laplacian and $\aa \geq 0$ is a constant. 
Since the reference configuration of the plate is flat, we may assume that it has a constant height equal to $H_0$, which depends on initial data. Writing $\eta = H_0+\zeta$, where $\zeta$ denotes the out-of-plane displacement of the plate relative to this reference configuration, it becomes evident that \eqref{EQPL} is equivalent to the standard description of the plate dynamics in elasticity theory in terms of~$\zeta$. 
Note that the case $\aa=0$ in \eqref{EQPL} corresponds to a perfectly elastic plate, while $\aa>0$ describes a linear elastic plate with structural square root damping which exhibits the empirically observed damping rates in elastic systems \cite{CR}. 
In~\eqref{EQPL}, $G (\bx,t)$ represents the external force combined with the fluid force acting on the structure in the vertical direction $e_3=(0,0,1)$. 
Namely,
\begin{equation}
	G
	=
	g
	-S_{\eta}
	(-p I+\nabla u) (\bx, \eta (\bx,t),t)
	n^{\eta}\cdot e_3
	\onon{\tt\times (0,T)},
\label{EQ82a}
\end{equation}
where $g\colon \tt \times (0,T)\to \mathbb{R}$ is the external force, $S_{\eta} (\bx,t)=\sqrt{1+|\nablax \eta (\bx,t)|^2}$ is the surface element of the deformed plate, and $n^{\eta}$ denotes the unit outward normal on $\Gamma(t)$, the graph of $\eta (\bx, t)$, given by
\begin{align}
n^{\eta}
= 
\frac{1}{\sqrt{1+  |\nablax \eta|^2}}
(-\partial_{1} \eta, -\partial_{2} \eta, 1).
\llabel{EQ83}
\end{align}
We assume the no-slip and the velocity matching conditions at the bottom and top of the fluid domain, respectively, i.e.,
\begin{align}
\begin{split}
&u(\bx,0,t)=0
\onon{\tt \times (0,T)},
\\&
u(\bx, \eta(\bx,t),t)
=(0,0,\etah_t)
\onon{\tt\times (0,T)}.
\label{EQ26}
\end{split}
\end{align}
Finally, we prescribe the initial conditions:
\begin{align}\label{IC}
u(\cdot,0)=u_0, \quad \eta(\cdot,0)=\eta_0, \quad \eta_t(\cdot,0)=\eta_1,
\end{align}
which satisfy the following compatibility conditions:
\begin{align*}
\diver u_0=0 \onon{\Omega_{\eta_0}}, \quad
u_0\cdot n^{\eta_0}=0 \onon{\Gamma_{\eta_0}}, \quad
u_0(\bx,0) =0 \onon{\tt}.
\end{align*}
\begin{Remark}
{
	From the fluid mechanical point of view, in the system \eqref{EQNS}--\eqref{IC} we adopt the Cauchy stress tensor  
	\begin{equation}
		\sigma(u,p)=-p \mathbb{I}_3 
		+ 
		\nabla u
		,
		\label{EQ83}
	\end{equation}
i.e., we omit the transposed gradient term~$\nabla^{T} u$.
	Following \cite[Remark 1]{GH}, we argue that $(\nabla^{T} u) n^{\eta}\cdot e_3=0$ on $\Gamma(t)$ and so that the term $G$ in~\eqref{EQ82a} remains unchanged even if the transposed term is included. 
	For the sake of completeness, we reproduce the argument here. 
	Differentiating the boundary condition $\uh_i(x,\etah(x,t),t)=0$, where $i=1,2$, with respect to horizontal variables gives
	\begin{align}
		\begin{split}
			\partial_1 \uh_1+\partial_3 \uh_1 \partial_1\etah=0,
			\\
			\partial_2 \uh_2+\partial_3 \uh_2 \partial_2\etah=0.
		\end{split}
		\llabel{EQ44}
	\end{align}
A straightforward calculation then yields 
	\begin{align}
		\begin{split}
			(\nabla^{T} u) n^\eta\cdot e_3
			&=
			\frac{1}{\sqrt{1+|\nabla_x \eta|^2}} 
			(-\partial_1\etah
			\partial_3 \uh_1
			-
			\partial_2\etah\partial_3 \uh_2
			+
			\partial_3 \uh_3)
			=
			\frac{\diver u}{\sqrt{1+|\nabla_x \eta|^2}}
			=
			0
			\llabel{EQ43}
		\end{split}
	\end{align} 
on~$\Gamma (t)$.
Hence, without loss of generality, we use the Cauchy stress tensor $\sigma (u,p)$ in~\eqref{EQ83}.
}
\end{Remark}

From the perspective of applied PDEs and fluid mechanics, contact between a deformable structure and a rigid boundary represents a fundamental singular regime in fluid-structure interaction, where the governing equations may lose ellipticity or even cease to be well-defined. 
In this paper, we address the question of whether such contact can occur between an elastic plate and the bottom of a fluid cavity in three space dimensions (see Figure~1). 

\begin{figure}
\includegraphics[width = 6cm]{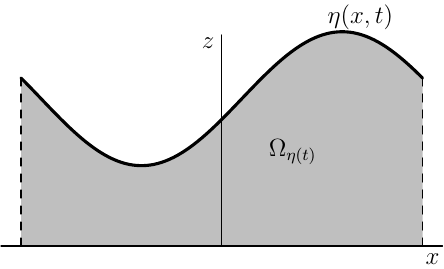} \hspace{1cm}
\includegraphics[width = 6cm]{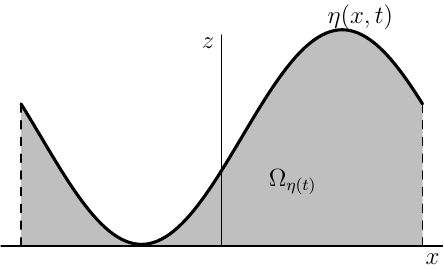}
\caption{2D cross sections of no-contact and contact configurations.} 
\end{figure}

\subsection{Brief literature review}\label{ssec01.2}
Many studies on the existence of solutions for fluid-structure interaction (FSI) problems focus on the motion of a rigid body immersed in a viscous incompressible fluid, governed by the Navier-Stokes equations; see, for instance, \cite{CST, DE1, DE2, Fe2, GM, SST, Ta, TT}.
A central difficulty in this setting is the possibility of body-body or body-boundary collisions: in particular, these existence results are valid up to a contact,
with a few exceptions \cite{CN,Fe1, SST}, where special weak solutions are constructed that persist even after collisions.
Consequently, to establish global-in-time existence of solutions without smallness assumptions on the initial data, one either has to introduce a suitable notion of solution beyond contact or prove that the contact never occurs. 
In the case of fluid-rigid body interaction, by now it is well-understood in which cases collision can be excluded; see, for instance, \cite{GVH,HT} and the references therein.
However, the question of whether contact can occur in fluid-structure interaction problems involving elastic structures remains largely open.
In two dimensions, Grandmont and Hillairet established in \cite{GH} global strong solutions for a viscoelastic beam-fluid interaction system by first ruling out contact and then propagating regularity.
More recently, in \cite{BR}, Breit and Roy  showed that contact can be excluded for weak solutions of the compressible beam-fluid system, provided certain additional regularity assumptions are satisfied.
It is worth noting that the approach in \cite{BR,GH} relies on a stream-function formulation of the fluid test function in two dimensions; see \cite[Remark~4.6]{BR}, where extending the no-contact result to three dimensions is posed as an open problem. In contrast, weak-solution frameworks that allow for contact have also been developed; see~\cite{CN,KMT}. 

In \cite{Le1,Le2}, for the same fluid-structure interaction problem considered in the present paper, the author proved global existence of strong solutions for small initial data and local-in-time existence of strong solutions for arbitrary initial data, in both two- and three-dimensional settings.
To obtain global strong solutions in three-dimensional space, one of the main difficulties is the possibility of a contact, where high regularity of the plate is needed to prevent it from occurring. 

Finally, we conclude the literature review by noting that results on fluid-beam systems in two dimensions, and, more generally, on fluid-shell systems in three dimensions, can be found in~\cite{Be,BKMT,CDEG,CGH,Gr1,Gr2,MC,MS}.


\subsection{The main result}\label{sec02}
The main goal of the present  paper is to establish a no-contact result for weak solutions $(u,\eta)$ under natural energy bounds together with an additional regularity assumption in three-dimensional space.
To the best of our knowledge, this is the first such no-contact result for a three-dimensional fluid-elastic structure interaction system.
Specifically, we obtain a uniform control of $\|\eta^{-1}(t)\|_{L^\infty(\mathbb T^2)}$ for all $t\in (0,T)$.
This answers the open question in \cite[Remark~4.6 (a)]{BR} in the affirmative.
The central ingredient of our analysis is a distance estimate (Proposition~\ref{P01}) that couples the fluid dissipation to a carefully designed divergence-free test function~$w$. Namely, we test the system with the couple $(\Delta_x \eta,w)$ which produces, after some cancellations and integrations by parts, a closed inequality for $\Vert \eta^{-1} (t)\Vert_{L^1 (\tt)}$,
which in turn yields the $L^\infty$ bound through a radial integral computation (see~Lemma~\ref{L03}).
Since the pair $(\Delta_x \eta,w)$ is not directly admissible as a test function in the weak formulation,
we justify its use by a smooth approximation scheme (see Appendix~\ref{sec04}).

The construction of $w$ is motivated by the shape of the fluid velocity in the lubrication approximation regime which describes close-to-contact dynamics of fluids~\cite{Sz}. The small relative thickness of the fluid height leads to a scale separation in which the pressure is depth-independent and the dominant part of the fluid velocity is the Poiseuille flow. Given the elastic structure on top of the fluid, the lubrication approximation regime can be described with a single nonlinear evolution equation of order six in spatial variables (cf.~\cite{BM} for 2D case where this lubrication approximation procedure has been rigorously justified). A key insight is that $\Vert \eta^{-1} (t)\Vert_{L^1 (\tt)}$ is a Lyapunov functional for this sixth-order equation, while here for the full FSI model and without smallness assumption we are able to control the growth of $\Vert \eta^{-1} (t)\Vert_{L^1 (\tt)}$.

First, we state the definition of the weak solution of the system \eqref{EQNS}--\eqref{IC}.
\begin{definition}
\label{D01}
(Weak solution)
Let $T>0$ and $(f,g,\eta_0,\eta_1, u_0)$ be such that 
\begin{align}
\begin{split}
&f\in L^2 ((0,T),L^2 (\Omega_{\eta})),\quad
g\in L^2 ((0,T), L^2 (\tt)),
\\&
\eta_0 \in H^2 (\tt) \text{~with~} \min_{x\in\tt} \eta_0(x) >0,\quad \eta_1 \in L^2 (\tt),\quad
u_0 \in L^2 (\Omega_{\eta_0}).
\end{split}
\llabel{EQ51}
\end{align}
We say that $(\eta,u)$ is a weak solution of
\eqref{EQNS}--\eqref{IC} if the following conditions are fulfilled:
\begin{enumerate}
\item For the structure height $\eta$, we have
\begin{align}
\eta \in W^{1,\infty} ((0,T), L^2 (\tt))
\cap L^\infty ((0,T), H^2 (\tt)) \quad 
\text{with~$ \eta>0$ a.e.},
\label{EQ11b}
\end{align}
$\eta (x,0)=\eta_0 (x)$, and $ \eta_t (x,0) = \eta_1 (x)$.

\item The fluid velocity $u$ satisfies
\begin{align}
u \in L^2 ((0,T), H^1 (\Omega_\eta)^3)
\cap 
L^\infty ((0,T), L^2 (\Omega_\eta)^3)
,
\label{EQ11a}
\end{align}
with $u(\bx, \eta (x,t),t) =  \eta_t (\bx,t) e_3$, $u(x,0,t)=0$, and $u(x,0)=u_0 (x)$.

\item 
For all test functions $(\phi, \Phi)$ with
\begin{align}
\begin{split}
	&
	\phi
\in 
W^{1,\infty}  ((0,T), L^2 (\tt))
\cap
L^\infty ((0,T), H^2 (\tt)),
\\&
\Phi \in
W^{1,2}((0,T), L^2 (\Omega_\eta)^3)
\cap
L^\infty ((0,T), L^2(\Omega_\eta)^3) 
\cap 
L^2 ((0,T), V_{\eta})
\end{split}
\label{EQ60a}
\end{align} 
where $V_{\eta} = \{u \in H^1 (\Omega_\eta)^3 : \diver u =0\}$,
and $\Phi (x,\eta (x,t),t) = \phi(x,t) e_3$,
the equality 
\begin{align}
\begin{split}
\frac{d}{dt}
\left(	\int_{\tt}
 \eta_t \phi 
+
\int_{\Omega_\eta}
u\cdot \Phi
\right)
&
= \int_{\Omega_\eta}
\left(	u  \cdot\Phi_t 
+
u \otimes u : \nabla \Phi
-\nabla u : \nabla \Phi
+f \cdot \Phi \right)
\\&
\quad
+ \int_{\tt}
\left(\eta_t \phi_t
+ 
g \phi
-
\Deltax \eta \Deltax \phi
+
\aa  \eta_t  \Deltax \phi
\right)
\label{EQweak}
\end{split}
\end{align}
holds in $\mathcal{D}'(0,T)$.
\end{enumerate}
\end{definition}

Note that since we study a moving-boundary problem, we work with Bochner spaces defined on time-dependent spatial domains. These spaces are standard in the analysis of fluid-structure interaction; see \cite[Section~1.3]{CDEG} for details.
We point out that the existence of weak solutions for closely related three-dimensional plate-fluid interaction model has been established in \cite[Theorem~1]{Gr1}, where the author considers the clamped boundary condition in the horizontal directions; see also \cite{TW, MC} for the existence of weak solutions for other fluid-structure interaction models.
Moreover, the authors of the present paper have established a global-in-time existence result of weak solutions for the same model considered here; see \cite[Proposition~3.1]{BKLM}.
We also note that, for the weak solutions constructed in those works, the following dichotomy holds: Either the solution exists globally in time and no contact occurs, or it exists only up to a finite time $T_\ast$ at which contact occurs. However, weak solutions are generally not unique; therefore, the fact that a particular construction avoids contact does not imply that contact is excluded for an arbitrary weak solution in the same class. Theorem~\ref{T01} rules out contact for any weak solution satisfying the additional regularity assumptions and, moreover, yields a uniform positive separation of the plate from the rigid boundary. By contrast, in the class of strong solutions one has existence and uniqueness (locally in time, or globally for small data), and such solutions do not exhibit contact; see, e.g.,~\cite{Le1,Le2}.

The following theorem is our main result.
\cole
\begin{Theorem}
\label{T01}
Let $\nu \geq 0$.
Assume that $(u,\eta)$ is a weak solution
to~\eqref{EQNS}--\eqref{IC}.
In addition, suppose that
\begin{align}
\begin{split}
&
u
\in
H^1 ((0,T), L^2 (\Omega_{\eta})^3),
\\&
\eta
\in
L^{\infty} ((0,T), H^{10/3} (\tt))
\cap H^1 ((0,T), H^1 (\tt))
.
\end{split}
\llabel{EQregularity2}
\end{align}
Then there exists a constant $C>0$, depending on the regularity of $(u,\eta)$
 and $(u_0,\eta_0,\eta_1)$,
	such that
	\begin{align}
		\sup_{t\in (0,T)}
		\Vert \eta^{-1} (t,\cdot)\Vert_{L^\infty (\tt)}
		\leq
		C.
		\label{EQnoContact}
	\end{align}
In particular, there is no contact on $[0,T]$.
\end{Theorem}
\colb

Note that the domain $\Omega_{\eta}$
in Definition~\ref{D01}
may degenerate at the times of the possible
contact;
this is handled by $T_{\ast}$ in the proof of the theorem below.

Formally testing \eqref{EQNS}$_1$ and \eqref{EQPL} with $u$ and $\eta_t$,
respectively,
integrating by parts in space, and integrating in time from $0$ to $t$, we arrive at
\begin{align}
	\begin{split}
		&
		\frac{1}{2}
		\left(
		\|u(t)\|^2_{L^2(\Omega_{\eta} (t))}
		+
		\|\Delta_{\bx}\etah(t)\|^2_{L^2(\tt)}
		+
		\|\partial_t\etah(t)\|^2_{L^2(\tt)} 
		\right)
		+
		\aa
		\int_0^t
		\int_{\tt}
		|\partial_t\nabla_{\bx}
		\eta|^2
		\\&\indeq
		+
		\int_0^t
		\int_{\Omega_{\etah}(s)}|\nabla u|^2
		=
		\frac{1}{2}
		\left(	\Vert u_0\Vert_{L^2 (\Omega_{\eta (0)})}^2
		+
		\|\etah_1\|^2_{L^2(\tt)} 
		+ 
		\|\Delta_{\bx} \etah_0\|^2_{L^2(\tt)}
		\right)
		\\&\qquad \qquad\qquad \qquad\qquad \qquad
		+
		\int_0^t \int_{\Omega_{\eta}(s)}
		\fNS\cdot u
		+
		\int_0^t \int_{\tt} 
		g \partial_t \eta,
		\label{EI}
	\end{split}
\end{align}
where we also used the Reynolds transport theorem.
The regularity of $(u,\eta)$ in~\eqref{EQ11b} and \eqref{EQ11a} is motivated by the basic energy bound \eqref{EI} when $f=g=0$.
\begin{Remark}
Using the incompressibility condition \eqref{EQNS}$_2$ and the kinematic coupling \eqref{EQ26}$_2$, we obtain
\begin{align}
\begin{split}
	&
	\frac{d}{d t}
	\int_{\mathbb{T}^2}
	\eta (\bx,t)
	=
	\int_{\mathbb{T}^2}
	\pt \eta (\bx,t)
	=
	\int_{\mathbb{T}^2} 
	u_3 (\bx, \eta (\bx,t), t)
	\\&
	=
	\int_{\partial {\Omega}_{\eta (t)} }
	u(\bx,z, t) \cdot n^\eta 
	=
	\int_{{\Omega}_{\eta (t)}} \diver u (\bx, z)
	=0,
	\llabel{EQcons}
\end{split}
\end{align}
which leads to
\begin{align}
	\bar{\eta}_0
	:=
	\int_{\mathbb{T}^2} \eta_0 (\bx) 
	=
	\int_{\mathbb{T}^2} \eta (\bx,t) 
	\comma 
	0\leq t < T,
   \llabel{EQ01a}
\end{align}
i.e., the mean value of the fluid height is constant in time.
\end{Remark}

Throughout the paper, we denote by $C\geq 1$ a sufficiently large constant which may depend on the regularity of the solution $(u,\eta)$ and $T$, and the value of $C$ may vary from line to line. 

This paper is organized as follows.
In Section~\ref{sec03}, we prove our main result by combining two main ingredients: the distance estimate (Proposition~\ref{P01}) and an interpolation lemma (Lemma~\ref{L03}). In the proof of the distance estimate, we construct the test function $w$, establish the key identity (see~\eqref{dist2}) linking fluid dissipation, plate bending, $\Vert \eta^{-1} (t) \Vert_{L^1 (\tt)}$, and then close the estimate via Gronwall's inequality. The smooth approximation procedure that permits inserting the pair $(\Deltax \eta, w)$ as a test function in the weak formulation is justified in Appendix~\ref{sec04}. Finally, in Appendix~\ref{sec05}, we collect two nontrivial auxiliary lemmas used in the proof of the distance estimate.

\section{Proof of Theorem~\ref{T01}}
\label{sec03}
The following statement is a key step in the proof of Theorem~\ref{T01}. The terminology \emph{distance estimate} is taken from \cite[Proposition~3]{GH}.

\cole
\begin{Proposition}[Distance estimate]
	\label{P01}
Let $\eta$ be a
function such that $\inf_{\tt\times(0,T)}\eta>0$.
Assume that $(u,\eta)$ is a weak solution
	to~\eqref{EQNS}--\eqref{IC}.
	In addition, suppose that $(u,\eta)$ satisfies the regularity assumed in Theorem~\ref{T01}, i.e.,
	\begin{align}
		\begin{split}
			&
			u
			\in
			H^1 ((0,T), L^2 (\Omega_{\eta})^3),
			\\&
			\eta
			\in
			L^{\infty} ((0,T), H^{10/3} (\tt))
			\cap H^1 ((0,T), H^1 (\tt))
			.
		\end{split}
		\label{EQregularity}
	\end{align}
	Then there exists a constant $C>0$ depending on the regularity of $(u,\eta)$ and $(u_0,\eta_0,\eta_1)$
	such that
	\begin{align}
		\sup_{t\in (0,T)}
		\Vert \eta^{-1} (t,\cdot)\Vert_{L^1 (\tt)}
		\leq
		C.
		\llabel{EQ61}
	\end{align}
\end{Proposition}
\colb

\begin{proof}
Define
\begin{align}
	&q_0 
	(x,t)
	= -\frac{6}{\eta^2 (x,t)}
	\label{def:q0}
\end{align}
and
\begin{align}
	\label{w_alpha}
	&w_\alpha(x, z,t) = \frac12 z(z-\eta(x, t))\pa_\alpha q_0,
	\qquad \alpha=1,2,
	\\&
	w_3 (x, z,t)=
	\frac{1}{12}
	z^2(3\eta-2z)\Deltax q_0 
	+ \frac14
	z^2\nablax\eta\cdot\nablax q_0.
	\label{w3}
\end{align}
{Note that the definition of $w$ resembles the fluid velocity in the
lubrication approximation regime and $q_0$ plays the role of the
pressure (see~\cite{BM} for the 2D case).} 
One may readily check that the test function $w= (w_1,w_2,w_3)$ defined by \eqref{w_alpha} and \eqref{w3} solves the
boundary value problem
\begin{align*}
\pa_i w_i &= 0,
\qquad \Omega_{\eta(t)} \times(0,T),
\\ 
w(x,0,t) &= 0,
\qquad x\in\tt,\ t\in (0,T),
\\
w(x, \eta(x, t),t) 
&
= \Deltax\eta e_3,
\qquad x\in\tt, 
\ t\in (0,T),
\end{align*}
and $q_0$ defined by \eqref{def:q0} satisfies \begin{align}\label{nabla_q}
\pa_\alpha q_0 = \pa_3^2w_\alpha
\comma
\alpha=1,2.
\end{align}
We define a pressure-like quantity as
\begin{align}
q(x,z,t) 
= 
q_0(x,t) 
+ 
\pa_3w_3(x,z,t).
\label{EQ100a}
\end{align}
By a smooth approximation, we may set $(\phi,\Phi)= (\Delta_x \eta, w)$ in~\eqref{EQweak} and obtain 
\begin{align}
\label{dist1}
\begin{split}
&
\int_0^t\int_{\Omega_\eta}
w u_t
+ 
\int_0^t\int_{\Omega_\eta} \nabla u : \nabla w
+ \int_0^t\int_{\tt}|\nablax\eta_t|^2
+
\int_0^t \int_{\tt}
(u \cdot \nabla u) w
\\&
- \int_0^t\int_{\tt}|\nablax\Deltax\eta|^2 +
\left.\int_{\tt} \left(-\frac{\aa}{2}|\Deltax\eta|^2 +
\eta_t\Deltax\eta\right)
\right|_{0}^t 
= \int_0^t\int_{\Omega_\eta}f  w
+
\int_0^t \int_{\tt} g \Deltax \eta
,
\end{split}
\end{align}
postponing the derivation to Appendix~\ref{sec04}.
Since $\div u=0$, we have
\begin{align}
\label{diss_term}
\begin{split}
	\int_0^t\int_{\Omega_\eta} \nabla u : \nabla w 
	&= \int_0^t\int_{\Omega_\eta}  (\nabla w - q\mathbb{I}_3): \nabla u
	\\&
	= \int_0^t\int_{\pa\Omega_\eta}(\nabla w - q\mathbb{I}_3) n^\eta\cdot u 
	- \int_0^t\int_{\Omega_\eta}\left(\Delta w - \nabla q\right)\cdot u.
\end{split}
\end{align}
Note that the first integral on the far right-hand side vanishes on the fixed part of the boundary $\tt \times \{z=0\}$. 
Using \eqref{EQ26}$_2$ and \eqref{EQ100a},
we obtain
\begin{align}
\label{int1}
\begin{split}
\int_0^t\int_{\pa\Omega_\eta}(\nabla w - q\mathbb{I}_3) n^\eta\cdot u
dS
= - \int_0^t \int_{\tt}
\left(\pa_1\eta\pa_1w_3 + \pa_2\eta\pa_2w_3 + q_0\right)\eta_t.
\end{split}
\end{align}
For the second integral on the far right-hand side of \eqref{diss_term}, we use \eqref{nabla_q} and \eqref{EQ100a} to obtain
\begin{align}
\begin{split}
	\int_0^t\int_{\Omega_\eta}\left(\Delta w - \nabla q\right)\cdot u &= \int_0^t\int_{\Omega_\eta}
	\begin{pmatrix}
	\pa_{11} w_1 + \pa_{22} w_1 + \pa_{33} w_1 - \pa_1q_0 - \pa_{13} w_3\\
	\pa_{11} w_2 + \pa_{22} w_2 + \pa_{33} w_2 - \pa_2 q_0 - \pa_{23} w_3\\
	\pa_{11} w_3 + \pa_{22} w_3 + \pa_{33} w_3 - \pa_{33} w_3
	\end{pmatrix}
        \cdot
        \begin{pmatrix}
	u_1\\u_2\\u_3
	\end{pmatrix}\\ 
	&= \int_0^t\int_{\Omega_\eta}
	\begin{pmatrix}
	\pa_{11} w_1 + \pa_{22} w_1 - \pa_{13} w_3\\
	\pa_{11} w_2 + \pa_{22} w_2 - \pa_{23} w_3\\
	\pa_{11} w_3 + \pa_{22} w_3
	\end{pmatrix}
        \cdot
	\begin{pmatrix}
	u_1\\u_2\\u_3
	\end{pmatrix}\\ 
	&= \int_0^t\int_{\Omega_\eta}\diver
	\begin{pmatrix}
	\pa_1 w_1 & \pa_2w_1 & -\pa_1w_3\\
	\pa_1 w_2 & \pa_2w_2 & -\pa_2w_3\\
	\pa_1 w_3 & \pa_2w_3 & 0
	\end{pmatrix}
       \cdot
   \begin{pmatrix}
	u_1\\u_2\\u_3
	\end{pmatrix}\\ 
	&=: \int_0^t\int_{\Omega_\eta}
	\diver A(w) \cdot u.
\label{int2}
\end{split}
\end{align}
Integrating by parts and employing the boundary conditions \eqref{EQ26}, we arrive at
\begin{align} 
\begin{split}
\label{int3}
\int_0^t\int_{\Omega_\eta}
\diver A(w) \cdot u &= \int_0^t\int_{\pa\Omega_\eta}
A(w) n^\eta\cdot
u - \int_0^t\int_{\Omega_\eta}
A(w) : \nabla u\\ 
&= - \int_0^t\int_{\tt}\left(\pa_1 w_3\pa_1\eta + \pa_2w_3\pa_2\eta\right)\eta_t - \int_0^t\int_{\Omega_\eta}
A(w) : \nabla u.
\end{split}
\end{align}
Inserting \eqref{int1}--\eqref{int3} into \eqref{diss_term}, we conclude that
\begin{align}
\begin{split}
&
\int_0^t\int_{\Omega_\eta} \nabla u : \nabla w = \int_0^t\int_{\Omega_\eta}
A(w) : \nabla u - \int_0^t\int_{\tt} q_0\eta_t
\\&
= \int_0^t\int_{\Omega_\eta}
A(w) : \nabla u + 6\int_0^t\int_{\tt} \frac{\eta_t}{\eta^2}
=  
\int_0^t\int_{\Omega_\eta}
A(w) : \nabla u 
- 
6\left.\int_{\tt} \frac{1}{\eta}\,
\right|_{0}^t.
\end{split}
\label{diss_term3}
\end{align} 
Combining \eqref{diss_term3} with~\eqref{dist1}, we arrive at
\begin{align}
\label{dist2}
\begin{split}
&
6\left.\int_{\tt} \frac{1}{\eta}\, 
\right|_{0}^t 
+ \int_0^t\int_{\tt}|\nablax\Deltax\eta|^2 
+ \left.\int_{\tt} \frac{\aa}{2}|\Deltax\eta|^2 
\right|_{0}^t 
\\&\quad
= \int_0^t\int_{\tt}|\nablax\eta_t|^2 
+
\left.\int_{\tt}  \eta_t \Deltax \eta \right|_{0}^{t}
+ 
\int_0^t\int_{\Omega_\eta}
A(w) : \nabla u 
+ 
\int_0^t\int_{\Omega_\eta}
w u_t
\\&\qquad
- 
\int_0^t\int_{\Omega_\eta} f  w
-
\int_0^t \int_{\tt} g \Deltax \eta
+\int_0^t \int_{\tt}
(u \cdot \nabla u) w
,
\\&\quad
=:
W_1 + W_2 + W_3+W_4+W_5+W_6+W_7.
\end{split}
\end{align}
Using the regularity of $\eta$ in~\eqref{EQ11b} and \eqref{EQregularity}, we have
\begin{align}
W_1+W_2
\leq 
C.
\label{EQ20a}
\end{align}
For the term $W_3$, using the Young's inequality, we arrive at
\begin{equation}
|W_3| 
\leq 
C\|A(w)\|_{L^2 L^2}^2 
+ 
C \|\nabla u\|_{L^2 L^2}^2
\leq
C \|A(w)\|_{L^2 L^2}^2
+C.
\label{EQW2}
\end{equation}
Recall that
\begin{equation*}
A(w) = \begin{pmatrix}
\pa_1 w_1 & \pa_2w_1 & -\pa_1w_3\\
\pa_1 w_2 & \pa_2w_2 & -\pa_2w_3\\
\pa_1 w_3 & \pa_2w_3 & 0
\end{pmatrix}.
\label{EQA}
\end{equation*}
Straightforward calculations of derivatives of~\eqref{w_alpha} and \eqref{w3} and integration in the vertical variable~$z$, reveal that
\begin{align}
\label{A22}
&
\int_0^t\int_{\Omega_\eta} |\nabla_x w_i|^2 
\leq 
C \int_0^t\int_{\tt}
\left(
\frac{|\nablax^2\eta|^2}{\eta}
+ \frac{ |\nablax\eta|^4}{\eta^3}
\right)
\end{align}
for $i=1,2$,
and
\begin{align}
\int_0^t\int_{\Omega_\eta} |\nabla_x w_3|^2 
\leq 
C \label{A3}\int_0^t
\int_{\tt}
\left( \eta |\nablax\Deltax\eta|^2
+
\frac{|\nablax^2\eta |^2 |\nablax\eta|^2}{\eta}
+ \frac{|\nablax\eta|^6}{\eta^3}
\right)
.
\end{align}
The first term on the right-hand side of \eqref{A22} may be estimated using \eqref{EQregularity}\textsubscript{2}
and the Sobolev embedding $H^{10/3} (\tt) \hookrightarrow W^{2,\infty} (\tt)$
as
\begin{equation*}
\int_0^t\int_{\tt}
\frac{|\nablax^2\eta|^2}{\eta} 
\leq 
C
\int_0^t\int_{\tt}\frac{1}{\eta}
,
\end{equation*}
while the second term on the right-hand side of \eqref{A22} may be estimated using Lemma~\ref{L01} as
\begin{align*}
\begin{split}
&\int_0^t\int_{\tt}
\frac{|\nablax \eta|^4}{\eta^3} 
\leq C\int_0^t\|\nablax^2\eta\|_{L^\infty (\tt)}^2
\int_{\tt} \frac{1}{\eta} 
\leq C\int_0^t\int_{\tt}
\frac{1}{\eta},
\end{split}
\end{align*}
using again
the Sobolev embedding $H^{10/3} (\tt) \hookrightarrow W^{2,\infty} (\tt)$.
Therefore, we conclude that for $i=1,2$, we have
\begin{equation}
\label{A22eps}
\int_0^t
\int_{\Omega_\eta} |\nabla_x w_i|^2 
\leq 
C \int_0^t\int_{\tt} \frac{1}{\eta}  .
\end{equation}
Next, we estimate the terms on the right-hand side of~\eqref{A3}.
For the first term, we have
\begin{align*}
\int_0^t\int_{\tt} \eta |\nablax\Deltax\eta|^2 
&
\leq 
\|\eta\|_{L^\infty L^\infty}
\|\nablax\Deltax\eta\|_{L^2 L^2}^2
\leq 
C.
\end{align*}
For the second and third terms, we use Lemma~\ref{L01} to obtain
\begin{align*}
\int_0^t\int_{\tt}
\frac{|\nablax^2 \eta|^2|\nablax\eta|^2}{\eta}&
\leq C\int_0^t\|\nablax^2\eta\|_{L^\infty}
\Vert \nablax^2 \eta\Vert_{L^2}^2
\leq 
C
\end{align*}
and
\begin{align*}
\int_0^t\int_{\tt} \frac{|\nablax\eta|^6}{\eta^3} \leq 
C \|\nablax^2\eta\|_{L^\infty L^\infty}^3 
\leq 
C.
\end{align*}
Therefore, we conclude that
\begin{equation}\label{A3eps}
\int_0^t\int_{\Omega_\eta} |\nabla w_3|^2 
\leq 
C.
\end{equation}
From \eqref{EQW2}--\eqref{A3} and \eqref{A22eps}--\eqref{A3eps}, it follows that
\begin{align}
|W_3| 
\leq 
C
+
C\int_0^t\int_{\tt}\frac{1}{\eta}.
\label{EQ20b}
\end{align}
The term $W_4$ is estimated using the Cauchy-Schwarz and Young's inequalities and the regularity of $u$ in~\eqref{EQregularity} as
\begin{align}
|W_4|
\leq
\int_0^t 
\Vert u_t\Vert_{L^2}
\Vert w\Vert_{L^2}
\leq
C+
\int_0^t \int_{\Omega_\eta}
|w|^2.
\label{EQ10b}
\end{align}
From the definition of $w$ in~\eqref{w_alpha} and \eqref{w3} and Lemma~\ref{L01}, we get
\begin{align}
\begin{split}
	&
	\int_{0}^{t}
	\int_{\Omega_\eta}|w|^2 \leq C\int_{0}^{t}\int_{\tt}
	\left(
	\frac{|\nablax\eta|^2}{\eta} + \frac{|\nablax\eta|^4}{\eta} + \eta|\Deltax\eta|^2\right)\\
	&\indeq
	\leq C\|\nablax^2\eta\|_{L^\infty L^\infty} 
	+ 
	C\|\nablax^2\eta\|_{
		L^\infty L^\infty}^2\|\eta\|_{L^\infty L^\infty} 
	+ 
	C\|\eta\|_{L^\infty L^\infty}
	\Vert \Deltax\eta \Vert_{L^2 L^2}^2
	\leq 
	C.
\label{EQ10c}
\end{split}
\end{align}
Combining \eqref{EQ10b} with~\eqref{EQ10c}, we obtain
\begin{align}
|W_4|
\leq 
C.
\label{EQ20c}
\end{align}
Similarly to \eqref{EQ10b} and \eqref{EQ10c}, the term $W_5$ is estimated as
\begin{align}
|W_5|
\leq
\int_{0}^{t}
\Vert f\Vert_{L^2}
\Vert w\Vert_{L^2}
\leq
C.
\label{EQ20d}
\end{align}
For the term $W_6$, we appeal to Cauchy-Schwarz inequality, yielding
\begin{align}
|W_6|
\leq
\Vert g\Vert_{L^2 L^2}
\Vert \Deltax \eta\Vert_{L^2 L^2}
\leq 
C.
\label{EQ20e}
\end{align}
Finally, we estimate the nonlinear term by following the approach in \cite[Proposition~3]{GH}:
\begin{align}\label{W7}
|W_{7}| 
&\leq \int_0^t\int_{\Omega_\eta}|(u\cdot\nabla)u\cdot w| \leq 
C\int_0^t\int_{\tt}\left(\int_0^\eta |u|^2\right)^{1/2}\left(\int_0^\eta |\nabla u|^2\right)^{1/2}\sup_{z\in(0,\eta)}|w|.
\end{align}
From \eqref{w_alpha} and \eqref{w3}, we have a pointwise estimate
\begin{equation}
	\sup_{z\in(0,\eta(x))}
	|w| 
	\leq 
	C\left(|\Deltax\eta| + \frac{|\nablax\eta|}{\eta} + \frac{|\nablax\eta|^2}{\eta}\right)
	\comma
	x \in\tt,
	\llabel{EQ63}
\end{equation}
which leads to three integral terms on the right-hand side of \eqref{W7},
\begin{align*}
W_{71} &=  \int_0^t\int_{\tt}\left(\int_0^\eta |u|^2\right)^{1/2}\left(\int_0^\eta |\nabla u|^2\right)^{1/2}|\Deltax\eta|  ,\\
W_{72} &=  \int_0^t\int_{\tt}\left(\int_0^\eta |u|^2\right)^{1/2}\left(\int_0^\eta |\nabla u|^2\right)^{1/2} \frac{|\nablax\eta|}{\eta}  ,\\
W_{73} &=  \int_0^t\int_{\tt}\left(\int_0^\eta |u|^2\right)^{1/2}\left(\int_0^\eta |\nabla u|^2\right)^{1/2} \frac{|\nablax\eta|^2}{\eta}  .
\end{align*}
First, we use the Cauchy-Schwarz inequality to get
\begin{align}
\begin{split}
W_{71} 
&
\leq 
\|\Deltax\eta\|_{L^\infty L^\infty}
\Vert u\Vert_{L^2 L^2}
\Vert \nabla u\Vert_{L^2 L^2}
\leq
C,
\label{EQ71}
\end{split}
\end{align}
where we used $\eta\in L^\infty W^{2,\infty}$.
For the term $W_{72}$, we employ Lemma~\ref{L02}
%
to get
\begin{align}
	W_{72} 
	\leq 
	C \int_0^t
	\|\nablax\eta\|_{L^\infty}
	\int_{\Omega_\eta}|\nabla u|^2 \leq C
	,
\label{EQ20f}
\end{align}
while the term $W_{73}$ is bounded analogously to $W_{72}$, yielding
\begin{align}
\begin{split}
W_{73}
\leq C
\int_{0}^{t}
\Vert \nablax \eta\Vert_{L^\infty}^2
\int_{\Omega_\eta}	|\nabla u|^2
\leq C.
\label{EQ20g}
\end{split}
\end{align}
Combining \eqref{dist2}--\eqref{EQ20a}, \eqref{EQ20b}, \eqref{EQ20c}--\eqref{W7}, and \eqref{EQ71}--\eqref{EQ20g}, we conclude that
\begin{equation*}
\int_{\tt}
\frac{1}{\eta(t)}  
\leq
C
+  C\int_0^t\int_{\tt}\frac{1}{\eta(s)}
\comma t\in(0,T),
\end{equation*}
which leads to
\begin{align}
\int_{\tt}
\frac{1}{\eta (t)}
\leq
C\comma
t\in (0,T),
\llabel{EQ64}
\end{align}
by the Gronwall’s inequality.
\end{proof}

\cole
\begin{lemma}
	\label{L03}
Let $\eta$ be a
function such that $\inf_{\tt}\eta>0$
with $\eta\in W^{2,\infty}(\tt)$ and $\eta^{-1}\in L^{1}(\tt)$. Then $\eta^{-1}\in L^{\infty}(\tt)$, and we have the estimate
\begin{equation}
    \|\eta^{-1}\|_{L^\infty(\tt)} \leq\frac{8}{M} \left(\exp\left(\frac{LM}{2\pi}\right)-1\right),
   \llabel{EQ01}
\end{equation}
where $L= \|\eta^{-1}\|_{L^1(\tt)}$ and $M = \|\eta\|_{W^{2,\infty}(\tt)}$.
\end{lemma}
\colb

\begin{proof}
For $M_1,M_2>0$ and $R>0$, we define
\begin{equation*}
\Psi(M_1,M_2;R) 
= 
\int_0^R\frac{r}{M_1 + M_2 r^2}\dd r.
\end{equation*}
Direct computation shows that
\begin{equation}
\Psi(M_1,M_2;R) 
= 
\frac{1}{2M_2}
\ln\left(1 + \frac{M_2R^2}{M_1}\right).
\label{EQ52}
\end{equation}
Since $\eta \in W^{2,\infty}(\tt)$ and $W^{2,\infty}(\tt) \hookrightarrow C^1 (\tt)$, we deduce that $\eta$ attains its minimum value $\eta_{\text{min}}>0$ at some point $x_0\in \tt$.
Then $\nabla\eta (x_0) = 0$ and the Taylor theorem implies
\begin{equation}
\eta(x) 
\leq 
\eta_{\text{min}} 
+
\frac{1}{2} \|\eta\|_{W^{2,\infty}}|x-x_0|^2
\comma
x\in \tt.
\label{EQ54}
\end{equation}
We set $R=1/2$ in $\Psi(M_1,M_2;R)$ and use the polar coordinates to obtain
\begin{equation}
\int_0^{1/2}\frac{r}{M_1 + M_2 r^2}\dd r 
= \frac{1}{2\pi}\int_{B(x_0,1/2)}\frac{1}{M_1 + M_2 |x-x_0|^2}
\dd x.
\label{EQ50}
\end{equation}
Using \eqref{EQ54}--\eqref{EQ50} and periodic extension of $\eta^{-1}$ to $B(x_0,1/2)$ if $B(x_0,1/2)\not\subseteq \tt$, we have
\begin{align}
\begin{split}
	L
	&=
	\int_{\tt}\frac{1}{\eta}\dd x \geq \int_{B(x_0,1/2)}\frac{1}{\eta}
	\geq \int_{B(x_0,1/2)}\frac{1}{\eta_{\text{min}} +  \frac{1}{2}  \|\eta\|_{W^{2,\infty}}|x-x_0|^2}
	\dd x
	\\&
	\geq 
	2\pi
	\int_{0}^{1/2}
	\frac{r}{\eta_{\text{min}}
	+
	M r^2/2
	}
	\dd r
	= 
	2\pi 
	\Psi(\eta_{\text{min}},M/2;1/2)
	.
\label{EQ53}
\end{split}
\end{align}
From \eqref{EQ52} and \eqref{EQ53}, it follows that
\begin{align*}
	\|\eta^{-1}\|_{L^\infty(\tt)} 
	= 
	\frac{1}{\eta_{\text{min}}}
	\leq 
	\frac{8}{M} \left(\exp\left(\frac{LM}{2\pi}\right)-1\right),
\end{align*}
and the proof of the lemma is concluded.
\end{proof}

\begin{proof}[Proof of Theorem~\ref{T01}]
Let $T_\ast\in (0,T)$ be any time which is strictly less than the possible
first contact.
Then $\eta$ is a strictly positive function on~$(0,T_\ast]$,
and we may apply
Proposition~\ref{P01} and Lemma~\ref{L03} to show that
\begin{align}
	\sup_{t\in (0,T_{\ast})}
	\Vert \eta^{-1} (\cdot,t)\Vert_{L^\infty (\tt)}
	\leq C
	,
	\label{EQ200a}
\end{align}
where $C$ is a constant with dependence as in the statement of
the theorem.

Since $T_\ast$ is an arbitrary time less than the possible first contact,
uniformity of the bound \eqref{EQ200a} implies that there is no
contact on $[0,T]$,
and the theorem is proven.
\end{proof}

\begin{Remark}[Comparison with the no-contact result from \cite{GH}]
The no-contact result in \cite{GH} is obtained for strong solutions of two-dimensional fluid-beam interaction system, where no additional regularity is needed to justify the couple of test functions $(\Delta_x\eta, w)$ in the distance estimate. Here, on the other hand, we work with weak solutions for
the three-dimensional fluid-plate system, for
which we need higher regularity assumption of~$\eta$.
Moreover, the interpolation lemma in \cite[Proposition~7]{GH} for the beam case requires $\eta\in H^2(\mathbb{T})$, which is a consequence of the energy estimate. Here, with the plate, we need $\eta\in W^{2,\infty}(\tt)$. One can easily construct a 2D-counterexample for \cite[Proposition~7]{GH} for the plate case. Pick $x_0\in\tt$ and $r>0$ such that $B(x_0,r)\subset \tt$. Then define $\eta_c(x) = |x-x_0|^{3/2}$ for $x\in B(x_0,r)$ and extend it outside of the ball to a smooth positive and periodic function. One can readily check that $\eta^{-1}_c\in L^1(\tt)$ and $\eta_c\in H^2(\tt)$, but obviously $\eta^{-1}_c\notin L^\infty(\tt)$.
\end{Remark}

\appendix

\section{~}
\subsection{Justification of the smooth approximation}
\label{sec04}
In this section we justify the use of the test pair $(\phi, \Phi) = (\Delta_x \eta, w)$ in the weak formulation \eqref{EQweak} in order to derive~\eqref{dist1}.
Recall that admissible test functions $(\phi, \Phi)$ must satisfy the regularity condition \eqref{EQ60a} and the coupling condition $\Phi (x,\eta (x,t), t) = \phi (x,t) e_3$.

Denote by $\sigma(u,p)=\nabla u - p\mathbb{I}_3$ the Cauchy stress tensor and $I=(0,T)$.
From \eqref{EQ11a} and regularity assumptions on $u$ in~\eqref{EQregularity}, we have
\begin{align}\label{EQ255}
	\diver \sigma 
	= 
	\partial_t u 
	+ 
	(u\cdot \nabla) u 
	- 
	f  \in L^{1}(I,L^{3/2}(\Omega_{\eta (t)})),
\end{align}
where we used the Sobolev embedding $H^1(\Omega_{\eta (t)}) \hookrightarrow L^6 (\Omega_{\eta (t)})$ to bound the nonlinear term.
Therefore, $\diver\sigma \in L^{3/2}(\Omega_{\eta(t)})$ for a.e.~$t\in I$. 
Since $\partial \Omega_{\eta (t)}$ is Lipschitz by the regularity \eqref{EQregularity}, the trace theorem \cite[Lemma~1.2.3, page 51]{So} implies
\begin{align}
	\sigma n^\eta
	\in 
	L^1 (I, W^{-2/3,3/2} (\partial \Omega_{\eta (t)})),
	\llabel{EQ65}
\end{align}
where $W^{-2/3,3/2} (\partial \Omega_{\eta (t)})$ is the dual space of $W^{2/3,3} (\partial \Omega_{\eta (t)})$. 
Moreover, for $\xi\in L^\infty (I, W^{1,3} (\Omega_{\eta (t)}))$, we have
\begin{align}
	\int_0^T
	\int_{\Omega_{\eta (t)}} (\diver\sigma)\xi
=
\int_{0}^T
	\langle \sigma n^\eta, 
	\xi|_{\partial \Omega_{\eta (t)}} 
	\rangle_{\partial \Omega_{\eta (t)}}
	-
	\int_{0}^T
	\int_{\Omega_{\eta (t)}} 
	\sigma: \nabla \xi, 
	\label{trace}
\end{align}
where
$\langle \sigma n^\eta, \xi|_{\partial \Omega_{\eta (t)}} \rangle_{\partial \Omega_{\eta (t)}}$ is well-defined in the sense of the generalized trace.
For any \\$\phi \in L^\infty (I,W^{1,3} (\tt))$ with $\int_{\tt} \phi (x,t) dx= 0$ for a.e.~$t\in I$, we define the extension to the fluid domain
\begin{align}
	U(x,z,t)
	=
	(0,0,  z \phi (x,t)/\eta (x,t)),
	\llabel{EQ66}
\end{align}
where $(x, z,t)\in \Omega_{\eta (t)} \times I$.
Then $U \in L^\infty (I, W^{1,3} (\Omega_{\eta (t)}))$ and by the trace theorem, we have
$U |_{\partial\Omega_{\eta(t)}}   \in L^\infty (I, W^{2/3,3} (\partial \Omega_{\eta (t)}))$.
Direct calculation shows that $\int_{\partial \Omega_{\eta (t)}} U \cdot n^\eta dS=0$ for a.e.~$t\in I$.
For $\eps,\delta>0$,
let $\rho_\eps: \tt \to \mathbb{R}_{+}$ and $\theta_\delta: \mathbb{R} \to \mathbb{R}_{+}$ be the standard radial
mollifiers which satisfy $\int_{\tt} \rho_\eps (x) dx=1$ and $\int_{\mathbb{R}} \theta_\delta (s) ds =1$.
We define the mollifications in space
\begin{align*}
	\zeta^\eps (x)
	= 
(	\rho_\eps \ast \zeta) (x)
=\int_{\tt} \rho_\eps (x-y) \zeta (y) dy
\end{align*}
and
\begin{align*}
	\zeta^{\eps, \eps} 
	= 
	{\rho}_\eps \ast \zeta^\eps
	={\rho}_\eps \ast (\rho_\eps \ast \zeta)
	.
\end{align*}
We define the mollification in time 
\begin{align*}
	\zeta_\delta (t)
	= 
	( \theta_\delta 
	 \ast \zeta) (t)
	 = 
	 \int_{B_\delta} 
	 \theta_\delta (s) \zeta (t-s) ds.
\end{align*}
It is easy to check that for any $\zeta \in L^2 (\tt)$, we have
\begin{align}
	\int_{\tt}
	\zeta^{\eps,\eps}
	(x) \zeta (x) dx
	=\int_{\tt} 
	|\zeta^{\eps} (x)|^2 dx.
	\label{EQ61}
\end{align}
For $\eps, \delta>0$,
we take $\phi =\Delta_x \eta$ and
 $\phi^\eps_\delta (x,t)= \Delta_x \eta^{\eps,\eps}_\delta (x,t)$,
where we extend $\eta(x,t)=\eta(x,0)$ for $t<0$ and $\eta (x,t)= \eta (x,T)$ for $t>T$. 
From \cite[Exercise III.3.5, page 176]{Ga}, 
there exists some $\Phi^\eps_\delta \in L^\infty (I, W^{1,3} (\Omega_{\eta (t)}))$ such that
\begin{align}
	\begin{split}
		&\diver \Phi^\eps_\delta
		= 0 \inon{in~$\Omega_{\eta (t)}$}
		\\&
		\Phi^\eps_\delta
		=U_\delta^\eps \inon{on~$\partial \Omega_{\eta (t)}$},
		\label{EQphi}
	\end{split}
\end{align}
where $U^\eps_\delta (x,z,t)= (0,0,z \phi^\eps_\delta (x,t)/ \eta(x,t))$. 
Similarly, we obtain a corresponding $\Phi$ from~$\phi$.
The construction of $\Phi^\eps_\delta$ is performed pointwise in time via a bounded linear operator. Consequently, $\Phi^\eps_\delta$ is measurable as a map from $I$ into the space $W^{1,3} (\Omega_{\eta(t)})$.
It is easy to see that $\Phi^\eps_\delta (x, \eta(x,t),t) = \phi^\eps_\delta e_3$.
Therefore, the pair $(\phi^\eps_\delta, \Phi^\eps_\delta)$ is admissible as test functions in \eqref{EQweak}, and we obtain
\begin{align}
\begin{split}
	&
		\left.
	\int_{\tt} \eta_t \phi_\delta^\eps \right|_{0}^T
	=
	\int_{0}^T \int_{\tt}
	\left( \eta_t (\phi^\eps_\delta)_t
	+ 
	g \phi^\eps_\delta
	-
	\Deltax \eta \Deltax \phi^\eps_\delta
	+
	\aa \eta_t  \Deltax \phi^\eps_\delta
	\right)
	\\&\qquad
	-
	\int_{0}^T 
	\int_{{\Omega}_\eta (t)}
	(\pt u
	+(u\cdot \nabla )u -f
	)\cdot \Phi^\eps_\delta
	-
	\int_{0}^T 
	\int_{{\Omega}_\eta (t)}
	\nabla u : \nabla \Phi^\eps_\delta
	,
	\label{EQ40a}
\end{split}
\end{align}
where we integrated in time and used integration by parts.
Inserting $\xi = \Phi^\eps_\delta$ in \eqref{trace}, we have
\begin{align}
	\begin{split}
		&
		\int_{0}^T
		\langle \sigma n^\eta,
		\Phi^\eps_\delta
		|_{\partial \Omega_{\eta (t)}} 
		\rangle_{
		\partial\Omega_{\eta(t)}} 
		=
				\int_{0}^T
		\int_{\Omega_{\eta(t)}}
		\Phi^\eps_\delta 
		\cdot \diver\sigma
		+
				\int_{0}^T
		\int_{\Omega_{\eta(t)}} \sigma : \nabla \Phi_\delta^\eps
		\\&\qquad
		=
				\int_{0}^T
				\int_{\Omega_{\eta(t)}}\left (\partial_t u 
		+ 
		(u\cdot \nabla) u - f\right )\cdot \Phi^\eps_\delta
		+ 
				\int_{0}^T
		\int_{\Omega_{\eta(t)}} (\nabla u - p\mathbb{I}_3):
		\nabla \Phi^\eps_\delta
		.
	\end{split}
	\label{EQ40}
\end{align}
%
%
From \eqref{EQ40a} and \eqref{EQ40} it follows that
\begin{align}\label{EQWP}
	\begin{split}
		&
		\int_{0}^T \int_{\tt}
		\left( 
		\eta_t (\phi^\eps_\delta)_t
		+ 
		g \phi^\eps_\delta
		-
		\Deltax \eta \Deltax \phi^\eps_\delta
		+
		\aa \eta_t  \Deltax \phi^\eps_\delta
		\right)
		\\&\qquad
		=
		\int_{0}^T
		\langle 
		\sigma n^\eta, \Phi_\delta^\eps
		|_{\partial \Omega_{\eta (t)}}
		\rangle_{\partial\Omega_\eta(t)} 
		+
		\left.
		\int_{\tt}  \eta_t \phi^\eps_\delta  \right|^T_0.
	\end{split}
\end{align}
Inserting $\phi^\eps_\delta (x,t)= \Deltax \eta^{\eps,\eps}_\delta (x,t)$ into \eqref{EQWP}, using the identity \eqref{EQ61}, and passing to the limit $\eps,\delta \to 0$, we obtain
\begin{align}
	\begin{split}
		\left.\int_{\tt}
		\left( \eta_t
		\Deltax \eta
		-
		\frac{\nu}{2}
		|\Deltax \eta|^2\right)
		\right|_{0}^T
		&=
		\int_{0}^T \int_{\tt}
		\left(
		-| \nablax\eta_t|^2 
		+ 
		g \Deltax \eta
		+
		|\nablax 
		\Deltax \eta|^2
		\right) 
		\\&\quad
		-
		\int_{0}^T
		\langle \sigma n^\eta, \Phi |_{\partial \Omega_{\eta (t)}}
		\rangle_{\partial\Omega_{\eta (t)} }.
	\end{split}
	\label{EQ12b}
\end{align}
The second term on the right-hand side of \eqref{EQ12b} is well-defined since $\Phi|_{\partial \Omega_\eta} \in L^\infty (I, W^{2/3,3}(\partial \Omega_\eta))$ and $\sigma n^\eta
\in 
L^1 (I, W^{-2/3,3/2} (\partial \Omega_{\eta (t)}))$.
It is easy to check that $w$ defined in \eqref{w_alpha}--\eqref{w3} satisfies $w\in L^\infty (I, W^{1,3} (\Omega_{\eta(t)}))$ since $\eta \in L^\infty (I, H^{10/3} (\tt))$.
Letting $\xi =w$ in~\eqref{trace}, we arrive at
\begin{align}
	\begin{split}
		\int_0^T
		\int_{\Omega_\eta}
		\diver \sigma \cdot w
		&=
		\int_{0}^{T}
		\langle
		\sigma n^\eta, w|_{\partial \Omega_\eta}
		\rangle_{\partial \Omega_\eta}
		-
		\int_0^T
		\int_{\Omega_\eta}
		(\nabla u -p\mathbb{I}_3) : \nabla w
		\\&
		=
		\int_0^T
		\int_{\Omega_\eta}
		(\partial_t u + (u\cdot \nabla )u - f) \cdot w.
	\end{split}
	\label{EQ251}
\end{align}
It is clear that $\Phi (x, \eta(x,t),t) = \Deltax \eta (x,t) e_3 = w(x,\eta(x,t),t)$.
Hence,
combining \eqref{EQ12b} with~\eqref{EQ251}, we obtain~\eqref{dist1}.

\subsection{Auxiliary lemmas}\label{sec05}

\begin{Lemma}
	\label{L01}
	There exists a constant $C>0$ such that 
	 for any positive function $\eta\in W^{2,\infty}(\tt)$, we have
	\begin{equation}
		|\nablax\eta(x)|^2
		\leq C\|\nablax^2\eta\|_{L^\infty(\tt)}\eta(x)
		\comma
		x\in\tt.
		\llabel{EQ60a}
	\end{equation}
\end{Lemma}

The proof is analogous to \cite[Proposition~6]{GH};
	for completeness, we provide it here.
\begin{proof}
	Let $z\in W^{2,\infty} (\tt)$ be a positive function.
	Since $z$ is periodic, we infer that $\partial_1 z$ attains its maximum and minimum on~$\tt$.
	Suppose the maximum is attained at $a\in \tt$.
	Then $\partial_{11} z (a) =0$ and hence
	\begin{align}
	\begin{split}
			|\partial_1 z (a)|^2
		&\leq
	|\partial_1 z (a)|^2
		+
		z
		\partial_{11} z (a)
		\leq
		\Vert |\partial_1 z|^2 + z \partial_{11} z\Vert_{L^\infty (\tt)}.
		\llabel{EQ61}
	\end{split}
	\end{align}
An analogous inequality holds at the point where $\partial_1 z$ attains its minimum.
Therefore, we have
\begin{align}
	|\partial_1 z (x)|
	\leq
	\Vert |\partial_1 z|^2 + z \partial_{11} z\Vert_{L^\infty (\tt)}^{1/2}
	\comma
	x\in \tt.
	\label{EQ60}
\end{align}
Now, set $z(x)=\sqrt{\eta (x)}$ for $x\in\tt$.
Since $\eta \in W^{2,\infty} (\tt)$ and $\eta>0$, it follows that $z\in W^{2,\infty} (\tt)$ and $z>0$.
Applying the estimate \eqref{EQ60} gives
\begin{align}
	|\partial_1 \eta (x)|
	\leq
	C
	\sqrt{\Vert \partial_{11} \eta\Vert_{L^\infty} \eta (x)}
	\comma
	x\in \tt.
	\label{EQ62}
\end{align}
Similar arguments yield
\begin{align}
	|\partial_2 \eta (x)|
	\leq
	C
	\sqrt{\Vert \partial_{22} \eta\Vert_{L^\infty} \eta (x)}
	\comma
	x\in \tt.
	\label{EQ63}
\end{align}
Combining \eqref{EQ62} and \eqref{EQ63} then completes the proof.
\end{proof}

We finally establish the following
Poincar\'e type inequality.
\begin{Lemma}
	\label{L02}
	Let $\etah\in L^{\infty}(\tt)$ with $\eta>0$ and
	$u\in H^1 (\Omega_{\eta})$ be such that
	$u|_{z=0}=0$. Then we have
	\begin{align}
		\llabel{EQ:Poincare}
		\left(\int_0^{\eta (x)} |u|^2\right)^{1/2}
		\leq 
		C\eta(x)
		\left(\int_0^{\eta (x)} |\nabla u|^2\right)^{1/2}
		\comma
		x \in \tt.
	\end{align} 
\end{Lemma}

\begin{proof}
Using the Cauchy-Schwarz inequality and boundary condition $u|_{z=0}=0$, we calculate
\begin{align*}
	&\int_{0}^{\etah(\bx)}
	u(\bx,z)^2 
	\dd z
	=
	\int_{0}^{\etah(\bx)}
	\left (\int_{0}^{z}
	\partial_3u(\bx,z')\dd z'
	\right )^2
	\dd z
	\\&\quad
	\leq 
	\int_{0}^{\eta (x)}
	\left( z\int_{0}^{z}
	(\partial_3 u)^2(\bx,z')
	\dd z'
	\right)
	\dd z
	\leq
	C \eta^2 (x) 
	\int_0^{\eta (x)}
	|\nabla u (x,z)|^2 \dd z,
\end{align*}
for any $x\in \tt$,
and the lemma is proven.
\end{proof}

\section*{Acknowledgments}
We sincerely thank the reviewers for the valuable feedback and suggestions which resulted in an improved version of the paper.
MB and BM were supported by the Croatian Science Foundation under project IP-2022-10-2962, while
IK was supported in part by the NSF grant DMS-2205493.


\begin{thebibliography}{[CDEG]}
\small

\bibitem[BKLM]{BKLM}
M.~Bukal, I.~Kukavica, L.~Li, and B.~Muha,
\emph{A global existence result on weak solutions for the 3D Navier-Stokes-plate system with no contact},
arXiv preprint 2511.22754 (2025).

\bibitem[BKMT]{BKMT}
A.~Balakrishna, I.~Kukavica, B.~Muha, and A.~Tuffaha,
\emph{Inviscid fluid interacting with a nonlinear two-dimensional plate}, Interfaces Free Bound. {\bf 27} (2025), no.~1, 141--175.

\bibitem[Be]{Be}
H. Beir\~ao~da~Veiga, \emph{On the existence of strong solutions to a coupled fluid-structure evolution problem}, 
J. Math. Fluid Mech. {\bf 6} (2004), no.~1, 21--52.

\bibitem[BR]{BR}
D.~Breit and A.~Roy,
\emph{Compressible fluids and elastic plates in 2D: A conditional no-contact theorem},
arXiv preprint arXiv:2411.01994 (2024).


\bibitem[BM]{BM} M.~Bukal and B.~Muha, 
\emph{Justification of a nonlinear sixth-order thin-film equation as the reduced model for a fluid-structure interaction problem}, 
Nonlinearity {\bf 35} (2022), no.~8, 4695--4726.


\bibitem[CDEG]{CDEG}
A.~Chambolle, B.~Desjardins, M.~J.~Esteban, and C.~Grandmont, \emph{Existence of weak solutions for the unsteady interaction of a viscous fluid with an elastic plate}, 
J. Math. Fluid Mech. {\bf 7} (2005), no.~3, 368--404.

\bibitem[CGH]{CGH}
J.-J.~Casanova, C.~Grandmont, and M.~Hillairet, 
\emph{On an existence theory for a fluid-beam problem encompassing possible contacts}, 
J. \'Ec. polytech. Math. {\bf 8} (2021), 933--971.


\bibitem[CN]{CN}
N.~V. Chemetov and \v S.~Ne\v casov\'a, 
\emph{The motion of the rigid body in the viscous fluid including collisions. Global solvability result}, 
Nonlinear Anal. Real World Appl. {\bf 34} (2017), 416--445.

\bibitem[CR]{CR} 
G.~Chen and D.~L.~Russell, \emph{A mathematical model for linear elastic systems with structural damping}, Quart.~Appl.~Math. {\bf 39} (1982), 433--454. 

\bibitem[CST]{CST}
C.~Conca, J.~A.~San Mart\'in, and M. Tucsnak,
\emph{Existence of solutions for the equations modelling the motion of a rigid body in a viscous fluid},
Comm. Partial Differential Equations {\bf 25} (2000), no.~5-6, 1019--1042.


\bibitem[DE1]{DE1}
B.~Desjardins and M.~J.~Esteban,
\emph{Existence of weak solutions for the motion of rigid bodies in a viscous fluid},
Arch. Ration. Mech. Anal. {\bf 146} (1999), no.~1, 59--71.


\bibitem[DE2]{DE2}
B.~Desjardins and M.~J.~Esteban, 
\emph{On Weak Solutions for Fluid‐Rigid Structure Interaction: Compressible and Incompressible Models}, 
Comm. Partial Differential Equations {\bf 25} (2000), no.~7-8, 263–285.


\bibitem[Fe1]{Fe1}
E.~Feireisl,
\emph{On the motion of rigid bodies in a viscous fluid}, 
Appl. Math. {\bf 47} (2002), no.~6, 463--484.


\bibitem[Fe2]{Fe2}
E. Feireisl,
\emph{On the motion of rigid bodies in a viscous incompressible fluid},
J. Evol. Equ. {\bf 3} (2003), no.~3, 419--441.


\bibitem[Ga]{Ga}
G.~P. Galdi, 
\emph{An introduction to the mathematical theory of the Navier-Stokes equations},
second edition, 
Springer Monographs in Mathematics, Springer, New York, 2011.

\bibitem[Gr1]{Gr1} 
C.~Grandmont,
\emph{Existence of weak solutions for the unsteady interaction of a viscous fluid with an elastic plate},
SIAM J.~Math.\ Anal.~\textbf{40} (2008), no.~2, 716--737.






\bibitem[Gr2]{Gr2}
C. Grandmont, \emph{Existence et unicit\'e{} de solutions d'un probl\`eme de couplage fluide-structure bidimensionnel stationnaire}, 
C. R. Acad. Sci. Paris S\'er. I Math. {\bf 326} (1998), no.~5, 651--656.



\bibitem[GH]{GH} 
C.~Grandmont and M.~Hillairet,
\emph{Existence of global strong solutions to a beam-fluid interaction system},
Arch. Ration. Mech. Anal. {\bf 220} (2016), no.~3, 1283--1333.



\bibitem[GM]{GM}
C.~Grandmont and Y.~Maday, 
\emph{Existence for an unsteady fluid-structure interaction problem}, 
M2AN Math. Model. Numer. Anal. {\bf 34} (2000), no.~3, 609--636.


\bibitem[GVH]{GVH}
D.~G\'erard-Varet and M.~Hillairet, 
\emph{Regularity issues in the problem of fluid structure interaction}, 
Arch. Ration. Mech. Anal. {\bf 195} (2010), no.~2, 375--407.


\bibitem[HT]{HT}
M.~Hillairet and T.~Takahashi, 
\emph{Collisions in three-dimensional fluid structure interaction problems}, 
SIAM J. Math. Anal. {\bf 40} (2009), no.~6, 2451--2477.

\bibitem[KMT]{KMT}
M.~Kampschulte, B.~Muha, and S.~Trifunovi\'c, 
\emph{Global weak solutions to a 3D/3D fluid-structure interaction problem including possible contacts}, 
J. Differential Equations {\bf 385} (2024), 280--324.


\bibitem[Le1]{Le1}
J. Lequeurre, 
\emph{Existence of strong solutions to a fluid-structure system}, 
SIAM J. Math. Anal. {\bf 43} (2011), no.~1, 389--410.

\bibitem[Le2]{Le2}
J.~Lequeurre, 
\emph{Existence of strong solutions for a system coupling the Navier-Stokes equations and a damped wave equation}, 
J. Math. Fluid Mech. {\bf 15} (2013), no.~2, 249--271.




\bibitem[MC]{MC}
B.~Muha and S.~\v Cani\'c, \emph{Existence of a weak solution to a nonlinear fluid-structure interaction problem modeling the flow of an incompressible, viscous fluid in a cylinder with deformable walls}, Arch. Ration. Mech. Anal. {\bf 207} (2013), no.~3, 919--968.



\bibitem[MS]{MS}
B.~Muha and S.~Schwarzacher, \emph{Existence and regularity of weak solutions for a fluid interacting with a non-linear shell in three dimensions}, 
Ann. Inst. H. Poincar\'e{} C Anal. Non Lin\'eaire {\bf 39} (2022), no.~6, 1369--1412.






\bibitem[SST]{SST}
J.~A. San~Mart\'in, V.~N. Starovoitov, and M. Tucsnak, 
\emph{Global weak solutions for the two-dimensional motion of several rigid bodies in an incompressible viscous fluid}, 
Arch. Ration. Mech. Anal. {\bf 161} (2002), no.~2, 113--147.


\bibitem[So]{So}
H. Sohr, 
\emph{The Navier-Stokes equations}, 
Birkh\"auser Advanced Texts: Basler Lehrb\"ucher, Birkh\"auser, Basel, 2001.


\bibitem[Sz]{Sz} 
A.~Z.~Szeri,
 \emph{Fluid Film Lubrication}, Cambridge University Press, Cambridge, 2012.


\bibitem[Ta]{Ta} 
T.~Takahashi,
\emph{Analysis of strong solutions for the equations modeling the motion of a rigid-fluid system in a bounded domain},
Adv. Differential Equations {\bf 8} (2003), no.~12, 1499--1532.







\bibitem[TT]{TT}
T.~Takahashi and M.~Tucsnak,
\emph{Global strong solutions for the two-dimensional motion of an infinite cylinder in a viscous fluid},
J. Math. Fluid Mech. {\bf 6} (2004), no.~1, 53--77.

\bibitem[TW]{TW} 
S.~Trifunovi\'{c} and Y.G.~Wang,
\emph{Existence of a weak solution to the fluid-structure interaction problem in 3D},
J.~Differential Equations~\textbf{268} (2020), no.~4, 1495--1531.




\end{thebibliography}
\end{document}